\documentclass[10pt]{amsart}   
\usepackage[utf8]{inputenc}
\usepackage[english]{babel}
\usepackage{mathrsfs}
\usepackage{array}
\usepackage{amssymb, amscd, amsmath, amsthm, latexsym, enumerate,amsfonts}
\usepackage{mathtools}
\usepackage{graphicx}
\usepackage[font=small,labelfont=bf]{caption}
\usepackage{float}
\usepackage{tikz}
\usetikzlibrary{matrix,shapes,arrows,positioning,chains, calc, decorations.pathmorphing}
\usepackage{dsfont}
\usetikzlibrary{matrix}
\usepackage{tikz-cd}
\usepackage{xfrac}
\usepackage{indentfirst}
\usepackage{hyperref, bookmark}
\usepackage{cleveref}
\usepackage{xcolor}
\usepackage{MnSymbol}
\usepackage{wrapfig}
\usepackage{faktor}
\usepackage{relsize}

\graphicspath{{Figs/}}

{\left\lbrace\begin{array}{@{}l@{}}}%
{\end{array}\right.}

\DeclareMathOperator{\Id}{Id}
\DeclareMathOperator{\id}{Id}

\newcommand{\R}{\mathbb{R}}

\newcommand{\Z}{\mathbb{Z}}
\newcommand{\C}{\mathbb{C}}

\makeatletter
\newcommand{\newreptheorem}[2]{\newtheorem*{rep@#1}{\rep@title}\newenvironment{rep#1}[1]{\def\rep@title{#2 \ref*{##1}}\begin{rep@#1}}{\end{rep@#1}}}
\newcommand{\littletaller}{\mathchoice{\vphantom{\big|}}{}{}{}}
\makeatother

\theoremstyle{plain}
\newtheorem{theorem}{Theorem}[section]
\newreptheorem{theorem}{Theorem}

\newtheorem{lemma}[theorem]{Lemma}
\newtheorem{prop}[theorem]{Proposition}
\newtheorem{cor}[theorem]{Corollary}
\theoremstyle{definition}
\newtheorem{defi}[theorem]{Definition}

\newtheoremstyle{remark-italic}
  {10pt}
  {10pt}
  {\itshape}
  {}
  {\bfseries}
  {.}
  { }
  {}

\theoremstyle{remark-italic}

\newtheorem{rem}[theorem]{Remark}
\newtheorem{notation}[theorem]{Notation}

\usepackage{slashed}
\newcommand\restr[2]{{
  \left.\kern-\nulldelimiterspace 
  #1 
  \littletaller 
  \right|_{#2} 
  }}

\newcommand{\bigslant}[2]{{\raisebox{.2em}{$#1$}\left/\raisebox{-.2em}{$#2$}\right.}}

\begin{document}
\title{On sections of Lefschetz fibrations and bundles over 2-complexes}

\author{Jonathan A. Hillman }
\address{School of Mathematics and Statistics,  University of Sydney\\
     NSW 2006, Australia}
\email{jonathanhillman47@gmail.com}
\author{Riccardo Pedrotti}
\address{University of Massachusetts, Amherst, MA 01003, USA}

\email{rpedrotti@umass.edu}

\begin{abstract}
We address the question of existence of sections of fibrations in two settings. First, we show that a bundle with base a finite 2-complex admits a section if and only if the inclusion of the fiber is $\pi_1$-injective and the as\-so\-cia\-ted short exact sequence of fundamental groups splits. Second, for Lefschetz fibrations over the disk we provide a complete algebraic criterion characterizing which loops in the boundary mapping torus extend to continuous or smooth sections over the disk. Finally, we apply our results to achiral Lefschetz fibrations over the sphere obtained by doubling along the vertical boundary, and give a criterion ensuring the existence of at least two homologically distinct sections. 
\end{abstract}



\maketitle
\tableofcontents
\section{Introduction}
Given $p: E \to B$ a fibration, a section is a map $s: B \to E$ such that
\[ p\circ s = \Id_B : B \to B.\]
For bundles over $2$-complexes, the existence of a section is essentially governed by homotopy theory as a consequence of the Homotopy Lifting Property (HLP). This is made precise in our first result.

\begin{reptheorem}{thm existence of section fiber bundle if monomorphism plus split exact sequence}
Let $p:E\to{B}$ be a fibre bundle with base $B$ a finite $2$-complex, and let $\ast_B$ and $\ast_E$ be base points such that $p(\ast_E)=\ast_B$. Then there is a map $s:B\to{E}$ such that $ps=id_B$ if and only if the inclusion of the fibre $F=p^{-1}(\ast_B)$ into $E$ induces a monomorphism from $\pi_1F$ to $\pi_1 E$ and there is a homomorphism $\mathfrak{s}:\pi_1B\to\pi_1E$ such that $p_*\mathfrak{s}=id_B$. If so, we may assume that $s_*=\mathfrak{s}$.
\end{reptheorem}

For Lefschetz fibrations, the question of the existence of a section is more ge\-o\-me\-tric in nature, since homotopy theory alone does not fully capture the topological structure of the spaces involved. At the time of writing, it is still unknown whether any (closed) Lefschetz fibration over the sphere admits a section (\cite[Problem 2.22]{BKR26}). A potential approach to solve this problem is the following:
\begin{enumerate}
    \item We divide the base sphere into two disks, $D^{\text{crit}} \cup_{\partial} D^{\text{triv}}$, one containing all the critical values of $\pi$ (the \textit{critical disk}) and the other not (the \textit{trivial disk}). They share a common boundary identified with the equator.
    \item After choosing a trivialization
\begin{equation*}\label{eq trivialisation for LF fibration}
    \Psi: \restr{E}{D^{\text{triv}}}\cong \Sigma \times D^{\text{triv}},
\end{equation*}  
there is no obstruction to finding a section over $D^{\text{triv}}$ as a preimage
\[ s^{\text{triv}}:= \Psi^{-1}\big(\{\ast \}\times D^{\text{triv}}\big)\]
\item We determine whether the loop $\restr{s^{\text{triv}}}{\partial}$ can be extended as a section over $D^{\text{crit}}$.
\end{enumerate}
In the second part of our paper, we completely solve step $(3)$, by fully characterizing the family of loops that can be extended to continuous or smooth sections over $D^{\text{crit}}$ (Thm. \ref{theorem criterion existence section LF over disk with several thimbles}).\\
Our algebraic criterion reads as follows. Let $\pi: E\to D$ be a Lefschetz fibration over the disk with $k$ critical values $z_1,\dots,z_k$, regular fiber a smooth closed orientable surface $\Sigma$ and monodromy $\tau \in \mathrm{Aut}(\Sigma)$. 
\begin{reptheorem}{theorem criterion existence section LF over disk with several thimbles}
    Let  $\pi: E\to D$ be a Lefschetz fibration as above. Let $\Sigma_{\tau}$ be the mapping torus of $\tau \in \mathrm{Aut}(\Sigma)$ identified with the boundary of $E$. Let $([\alpha],1)\in \pi_1(\Sigma_{\tau},\ast_{\Sigma})$ and let $s_{\alpha}$ be a section representative for it.\\ There exists a continuous {section} $\sigma_{\alpha}$ of $E$ extending $s_{\alpha}$ if and only if the following condition is met:\\
    choose a compatible factorization of the monodromy $\tau$:
    \[\tau = \tau_{V_k}\circ \cdots \circ \tau_{V_1}.\]
    Then, as a based loop in $\Sigma$, $\alpha$ can be decomposed as
    \[
    \alpha \sim \alpha_1 \star \tau_{V_1}^{-1}\alpha_2 \star \cdots \star (\tau_{V_{k-1}}\circ \cdots \circ \tau_{V_1})^{-1}\alpha_k,
    \]
    where each $\alpha_i$ is ``twisted-conjugated’’ to some multiple of a based representative $\widetilde{V_i}$ of the $i$-th vanishing cycle $V_i$. I.e. there exist an integer $m_i$ and a based loop $\zeta_i$ in $\Sigma$ such that
    \[
    \alpha_i \sim \zeta_i \star \widetilde{V_i}^{m_i} \star \tau_{V_i}^{-1}(\overline{\zeta_i}).
    \]
    Moreover, the section $\sigma_{\alpha}$ is smoothable if and only if $\forall i \ m_i \in \{0,1\}$.
\end{reptheorem}

We remark that a compatible positive factorization of the monodromy, arising for example after choosing an ordered set of vanishing paths, is an auxiliary choice in our theorem. If the loop $\alpha$ can be successfully decomposed with respect to one positive factorization, then it can be decomposed for any other factorization obtained via Hurwitz moves. Therefore, the property of being \textit{extendable} is a property of $\alpha$, and not of the specific choice of vanishing paths (see Rem. \ref{rem theorem independent on choice of vanishing paths}).\\

In \cite[Theorem 1.1]{Gom25}, the author constructed an example of an achiral Lefschetz fibration whose fiber $\Sigma$ has arbitrarily high genus, as a double of a Lefschetz fibration $f_0: X_0\to S^2$ with \textit{horizontal} boundary $\partial X=\sqcup S^3$ and such that $f_0$ does not admit any sections. The argument relies primarily on observing that $f_0$ does not admit a section for homology reasons: the (relative) fundamental class of the fiber is not primitive in $H_2(X_0)$ and the restriction $f_0: \partial X_0 \to S^2$ is a union of several Hopf fibrations. By doubling $X_0$, the author obtains an example of a (closed) \textit{achiral} Lefschetz fibration $D(X_0)\to S^2$ with no section.\\ In this paper we start with a Lefschetz fibration over the disk $E\to D$ and regular (closed) fiber $\Sigma$ and we double along its \textit{vertical} boundary. We obtain an achiral Lefschetz fibration $D(E)\to S^2$ with many sections since both $E$ and the gluing locus (a mapping torus), poses no homological obstructions. Our last result in this paper is the following criterion for these sections to be homologically essential. 
\begin{reptheorem}{theorem sections of doubles of LF over disk}
If there exists an index $j\in \{1,\dots,k\}$ and an integer $m$  such that $m\cdot [V_j] \notin \mathrm{Im}(\Id_\ast-\tau_*)\subset H_1(\Sigma;\Z)$, then $D(E)$ admits at least two homologically distinct sections.
\end{reptheorem}
\noindent We conclude by remarking that if the total monodromy $\tau$ of $E\to D$ is trivial, then the criterion of Theorem \ref{theorem sections of doubles of LF over disk} is automatically satisfied (see Cor. \ref{cor on theorem section doubles})\\

\noindent\textbf{Acknowledgments.}  We are grateful to {\.I}nan{\c{c}} Baykur for his careful explanations and insightful conversations about Lefschetz fibrations, and to Bob Gompf for several helpful discussions on his recent work.

\section{Sections of bundles over $2$-complexes}
A bundle projection $p:E\to B$ with base $B$ and fibre $F$ aspherical closed surfaces has a section if and only if the associated short exact sequence of fundamental groups splits. This is an immediate consequence of the Homotopy Lifting Property (HLP). For since $E$ is aspherical any splitting homomorphism is realized by a based map $s:B\to E$ such that $ps\simeq id_B$, and $s$ is then homotopic to a section, by the HLP. 

\begin{theorem}\label{thm existence of section fiber bundle if monomorphism plus split exact sequence}
Let $p:E\to{B}$ be a fibre bundle with base $B$ a finite $2$-complex, and let $\ast_B$ and $\ast_E$ be base points such that $p(\ast_E)=\ast_B$. Then there is a map $s:B\to{E}$ such that $ps=id_B$ if and only if the inclusion of the fibre $F=p^{-1}(\ast_B)$ into $E$ induces a monomorphism from $\pi_1F$ to $\pi_1 E$ and there is a homomorphism $\mathfrak{s}:\pi_1B\to\pi_1E$ such that $p_*\mathfrak{s}=id_B$. If so, we may assume that $s_*=\mathfrak{s}$.
\end{theorem}

\begin{proof}
The necessity of these conditions follows from the exact sequence of homotopy for the bundle If $p$ has a section $s$, so that $ps=id_B$, then $\pi_2(p)$ is an epimorphism, so $\pi_1F\to\pi_1E$ is injective, and $\mathfrak{s}=s_*$ is a splitting homomorphism.\\
Suppose that $\pi_1F\to\pi_1E$ is a monomorphism and $\mathfrak{s}:\pi_1B\to\pi_1E$ is a splitting homomorphism for $p_*$. Let $B^{[1]}$ be the 1-skeleton of $B$, and let $j:B^{[1]}\to{B}$ be the inclusion, Then $B^{[1]}\simeq\vee^rS^1$ for some $r\geq 0$, where each copy of $S^1$ represents a generator $a_i$ of $\pi_1(B)$.\\
For each such generator $a_i$, choose smooth (based) representatives $\widetilde{\gamma}_{a_i}$ of $\mathfrak{s}(a_i)$ in $E$. Since 
        \[ p\widetilde{\gamma}_{a_i}\sim a_i \ \text{ in } B,\] there is a smooth homotopy $H_{a_i}$ such that $H_{a_i}(-,0)=p\widetilde{\gamma}_{a_i}$ and $H_{a_i}(-,1)=a_i$. Bundle projections have the Homotopy Lifting Property \cite[Corollary 2.7.14]{Spa}, so we can solve the following lifting problem:
        \begin{center}
            \begin{tikzcd}
                S^1\times \{0\} \arrow{d} \arrow{r}{\widetilde{\gamma}_{a_i}} & E \arrow{d}{p} \\ S^1\times [0,1] \arrow{r}{H_{a_i}} \arrow[dashed]{ur}{\widetilde{H}_{a_i}}& B
            \end{tikzcd}
        \end{center}
        We have then a new homotopy $\widetilde{H}_{a_i}$ and by setting 
        \[\restr{s}{a_i}(-):=\widetilde{H}_{a_i}(-,1)\] 
        we have the required section over $a_i$. Let $s^{[1]}:B^{{1}}\to E$ be the section over the $1$-skeleton obtained by gluing together over $\ast_B$ the sections overs the single loops. By construction \[f_*\circ s^{[1]}_*=\mathfrak{s}\circ j_*,\] where $f$ is the inclusion of $\restr{E}{B^{[1]}}$ into $E$.\\
We shall induct on the number of 2-cells in $B$. Suppose that $B=B^{[1]}\cup_\chi{D^2}$, where $\chi : D^2\to B$ is the characteristic map of the $2$-cell.\\
Let $D=\tfrac{1}{2}D^2$ be a concentric smaller disc. Then $B^{[1]}$ is a deformation retract of $\overline{B\setminus{D}}$, and we may again use the HLP to extend $s^{[1]}$ over $\overline{B\setminus{D}}$. Thus by replacing $B^{[1]}$ by $\overline{B\setminus{D}}$ we may assume that $\chi$ is injective, and so $B=B^{[1]}\cup{D}$, where  $B^{[1]}\cap{D}\cong S^1$. We may also assume that $\ast\in\partial{D}=B^{[1]}\cap{D}$. Let $E_{[1]}:=p^{-1}(B^{[1]})$ and $E_D:=p^{-1}(D^2)$ be the restrictions of the bundle over $B^{[1]}$and $D$, respectively. Since $D$ is contractible, $E_D\to D$ is trivial and so there is a homeomorphism $h:E_D\cong{F\times{D}}$ such that $\text{proj}_2h=\restr{p}{E_D}$. $\partial{E_{[1]}}$ coincides with the restriction of $E$ over $\partial D$. Let $h_\partial=\restr{h}{\partial{E_{[1]}}}$.\\
Consider the Van Kampen pushout diagram for $E=E_{[1]}\cup_{h_\partial}F\times{D}$:\\
\begin{equation}\label{eq ses fund groups fibration with split}
\begin{CD}
\pi_1\partial{E_{[1]}}@>h_{\partial*}>>\pi_1F\\
@VVV @VVV\\
\pi_1E_{[1]}@>>>\pi_1E.
\end{CD}
\end{equation}
The element $\restr{s}{\partial D}\in\pi_1\partial{E_{[1]}}$ has image $1$ in $\pi_1E$, since $s$ splits the projection of $\pi_1E$ onto $\pi_1B$. Hence its image in $\pi_1F$ must be trivial, since the inclusion of $F$ into $E$ is injective on $\pi_1$. The map $\restr{s^{[1]}}{\partial{D}}:S^1=\partial{D}\to{F}\times{S^1}$ has the form $b\mapsto(k(b),b)$, for all $b\in\partial{D}$, and the Van Kampen argument shows that $k$ is homotopic to a constant map $c$. Let $k_t$ be a homotopy from $k_0=k$ to the constant map $k_1=c$. We may use such a null-homotopy of $k$ to extend the section across a collar neighborhood $N=\partial{D}\times[0,1]$ of $\partial{D}$ in $D$ by setting $s(b,t)=(k_t(b),t,b)$ for all $b\in\partial{D}$ and $0\leq{t}\leq 1$. We then set $s(d)=(c,d)$ for all $d$ in $D\setminus{N}$. This completes the inductive step.\\ The final assertion is clear from the construction of $s$.
\end{proof}

The homomorphism $\pi_1F\to\pi_1E$ is always a monomorphism when $B$ is aspherical. This is also the case if $\chi(F)<0$, for the image of $\pi_2B$ in $\pi_1F$ under the connecting homomorphism  of the long exact sequence of homotopy is central \cite{Go65}. One the other hand, if $\eta:S^3\to{S^2}$ is the Hopf bundle, then the composite with $pr_1:S^3\times{S^1}$ gives  a $T$-bundle over $B=S^2$ which has no section.\\
The following simple example shows that having a section is very much a property of the bundle map, and is not determined by the total space alone. The flat 4-manifold group $\pi$ with holonomy $(\mathbb{Z}/2\mathbb{Z})^2$ and presentation
\[
\langle{a,b,x,y}\mid{aba^{-1}=b^{-1}},~yay^{-1}=ab,~xyx^{-1}=y^{-1}b^{-1},
\]
\[
ax=xa,~bx=xb,~by=yb\rangle
\]
is the group of a $Kb$ bundle ($\pi_1F=\langle{a,b}\rangle$) over $Kb$ ($\pi_1B$ generated by the image of $\{x,y\}$), and as such has no section, since the action $\theta:\pi_1B\to{Out(\pi_1F)}$ does not lift to a homomorphism into $Aut(\pi_1F)$. However it is also the group of a $T$-bundle ($\pi_1F=\langle{b,y}\rangle$) over $T$ ($\pi_1B=\langle{a,x}\rangle$) which does have a section.

\begin{rem}\label{rem alternative proof of nullhomotopy section over boundary 2 cell}
    The crux of the proof of Thm. \ref{thm existence of section fiber bundle if monomorphism plus split exact sequence} is to prove that $\restr{s}{\partial D}$ is null-homotopic in $E_D\cong F\times D$. If $B$ is aspherical, we have the following more concise construction. By \cite[Cor. 2.2.3, Cor 2.2.5]{MP12} \[\pi_1(\chi^*E)\cong \pi_1E \times_{\pi_1B} \pi_1D^2, \] hence it fits in the following diagram:
    \begin{center}
        \begin{tikzcd}
        \pi_1S^1 \arrow[ddr,bend right,"\mathfrak{s}\circ \restr{\chi}{\partial}_*"'] \arrow[drr,bend left,"\text{incl.}_*"] \arrow[dr,dashed] \\
        & \pi_1(\chi^*E) \arrow[r] \arrow[d] & \pi_1D^2 \arrow[d, "\chi_*"] \\
        & \pi_1 E \arrow[r, "p_*"] & \pi_1B 
        \end{tikzcd}
    \end{center}
    The map induced by the universal property coincides with $\restr{s}{\partial D}$ under the isomorphism 
    \[E_D\cong \chi^*E,\] 
    and it is null-homotopic as the two maps inducing it are. Let $\widetilde{h}: D^2\to \chi^*E$ be an extension for $\restr{s}{\partial D}$. We can either proceed as in the proof of Thm. \ref{thm existence of section fiber bundle if monomorphism plus split exact sequence} to prove that $\widetilde{h}$ can be realized as a section or we can apply the Homotopy Extension Lifting Property (HELP) as follows. Since $\pi_2B=0,$ there exists an homotopy of homotopies realizing
        \[ H_{\chi}: p\widetilde{h}\sim \chi \]
     By HELP we can solve the following lifting problem:
        \begin{center}
            \begin{tikzcd}[column sep=3cm]
                D^2\times \{0\} \cup S^1\times [0,1] \arrow{d} \arrow{r}{\widetilde{h}\cup \restr{s}{\partial D}\circ \pi_1} & E \arrow{d}{p} \\ D^2\times [0,1] \arrow{r}{H_{\chi}} \arrow[dashed]{ur}{\widetilde{H}_{\chi}}& B
            \end{tikzcd}    
        \end{center}
    Let $\widetilde{h'}:= \widetilde{H}_{\chi}(-,1) :D^2\to E$, by construction it is a lift of the characteristic map $\chi$ coinciding with $\restr{s}{\partial D}$ over $S^1=\partial D$ and hence a section extending $\restr{s}{\partial D}$ over the whole $D$ which can then be pushed-forward to $E$. 
\end{rem}

In general, criteria for sections of bundles over bases of dimension $>2$ are not determined by fundamental group considerations alone. Left multiplication by the unit quaternions defines a “quaternionic Hopf fibration” from the unit sphere $S^7\subset\mathbb{H}^2$ to the quaternionic projective line $\mathbb{HP}^1\cong S^4$. No such bundle can have a section, since $\pi_4S^7=0$.\\
Good candidates for bundles over $S^3$ with no section are less well known, since $\pi_2(G)=0$ for any Lie group $G$, and so bundles over $S^3$ with structure group $G$ are trivial. Let $U$ be the group of unitary operators in the Hilbert space $\ell^2_\infty(\mathbb{C})$. Then $U$ is weakly contractible \cite{Ku65}, i.e., $\pi_iU=0$ for all $i\geq0$. Let $G=U/S^1$, where $S^1$ is the subgroup of scalar multiplications. Then $G$ is a (connected) topological group, since $S^1$ is a closed central subgroup, and so has a (connected) classifying space $BG$. We have $\pi_iBG\cong\pi_{i-1}G$ for all $i\geq1$ and $\pi_jG
\cong\pi_{j-1}S^1$, for all $j\geq1$. Hence $\pi_3BG\cong\Z$ and $\pi_iBG=0$ if $i\not=3$. Let $q:EG\to{BG}$ be the universal principal $G$-bundle, and let $p:E\to{S^3}$ be the pullback of $q$ over a generator of $\pi_3(BG)$. It is easy to see from the pullback square that the connecting homomorphism $\partial:\pi_3S^3\to\pi_2\Omega{K}$ in the exact sequence of homotopy is an isomorphism. Hence the bundle $p$ has no section. Can this construction be modified to give a fibre bundle over $S^3$ with fibre a compact manifold (perhaps $\mathbb{CP}^k$ for some $k\geq2$) and which has no section?

Let $M=S^1\times{S^2}$.
A. Hatcher has shown that $\mathrm{Diff}(M)\simeq{O(1)}\times{O(2)}\times\Omega{SO(3)}$ \cite{Ha81}.
This suggests the following example of an $M$-bundle over $S^3$ with structure group $G=\Omega SO(3)$.
View $S^1$ as $[-1,1]/\{-1\}\sim\{1\}$ and $S^3$ as $[-1,1]\times{S^2}$
with each of $\{-1\}\times{S^2}$ and $\{1\}\times{S^2}$ collapsed to a point.
Let $\varphi:S^3\to{SO(3)}$ be the double cover and define a function 
$F:S^2\to {\mathrm{Diff}(M)}$ by $F(z)([t],w)=([t],\varphi([t],z)(w))$, 
for all $-1\leq{t}\leq1$ and $w,z\in{S^2}$.
Is the $M$-bundle over $S^3$ with clutching function $F$ the bundle classified
by a map representing a generator of $\pi_3BG\cong\Z$?
Does this bundle have a section?
\section{Lefschetz fibrations over a surface}
\begin{defi}
    Let $X$ be a compact connected oriented smooth $4$-manifold and let $B$ be a connected oriented surface.\textit{ A Lefschetz fibration} on $X$ over $B$ with closed fiber is a smooth map
    \[ \pi : X\to B\]
    such that $\pi^{-1}(\partial B)=\partial X$. If $q\in X$ is a critical point for $\pi$, then $q\in \mathrm{int}X$ and it admits a neighborhood $U$ (called \textit{holomorphic} neighborhood) with complex coordinates $(z_1,z_2)\in \C^2$ compatible with the orientation of $X$ and $\pi(q)$ admits a neighborhood $V$ with complex coordinate $z\in \C$ compatible with the orientation of $B$ such that the restriction
    \[ \restr{\pi}{U} : U \to V\]
    is of the form \[ \pi(z_1,z_2)=z_1^2+z_2^2.\]
    We call $q$ a \textit{positive} critical point.\\
   If the local model for $\pi$ at $q$ is not orientation preserving, we call such a critical point a \textit{negative} critical point. If $\pi$ admits at least one negative critical point we call such map an \textit{achiral} Lefschetz fibration.
\end{defi}
As a consequence of the definition, the set $ X^{\text{crit}}\subset X$ of critical points is finite in both the chiral and achiral case, and so is $B^{\text{crit}}= \pi(X^{\text{crit}})$ the set of critical \textit{values} in $B$. We can always assume that each critical fiber (also called \textit{singular} fiber) contains only one critical point. For details the reader is referred to \cite{GS99}.

\begin{defi}
    A \textit{symplectic} Lefschetz fibration is a Lefschetz fibration equipped with a closed two form $\Omega\in \Omega^2(X)$ such that for each $x\in X$ it is non-degenerate on the vertical tangent space $\ker(d\pi_x)$, and which is a K\"ahler form for the (integrable) almost complex structure supported in some neighborhood of $X^{\text{crit}}$.
\end{defi}

\begin{theorem}\cite{GS99} Given a Lefschetz fibration $\pi: X \to B$ with fiber $F$. If $[F]\neq 0\in H_2(X;\R)$ then $X$ admits the structure of a symplectic Lefschetz fibration. In particular, if the genus of the fiber is at least $2$, $f$ is a symplectic Lefschetz fibration.  
\end{theorem}
\begin{rem} 
    Achiral Lefschetz fibrations are not necessarily symplectic and they do not necessarily support an almost complex structure on their total space.
\end{rem}
In a symplectic Lefschetz fibration $\pi: X \to B$, given a path $\psi: [0,1]\to B$ starting at a critical value $\psi(0)\in B^{\text{crit}}$ and avoiding $B^{\text{crit}}$ otherwise, we can define the thimble $\Delta_\psi$ as the set of points in $X$ that are sent by symplectic parallel transport to $x_0\in X^{\text{crit}}\cap \pi^{-1}(\psi(0))$. By working in holomorphic neighborhoods $U,V$ around $x_0=(0,0)$ and $\pi(x_0)=0$, the thimble $\Delta_\psi$ satisfies
\begin{equation}\label{eq thimble in local coordinates}
     \Delta_\psi=\bigcup_{t\in (0,1]}V_{\psi(t)}\cup V_0,
\end{equation}
where \begin{equation}\label{eq thimble WITH NO CRITICAL VALUE in local coordinates}
   V_{\psi(t)} = \sqrt{\psi(t)} S^1 = \left\{ \big(\sqrt{\psi(t)} y_1,\sqrt{\psi(t)}y_{2}\big) : (y_1,y_2) \in S^1 \subset \mathbb{R}^{2} \right\} \subset q^{-1}(t),
\end{equation}
for $\psi(t) \neq 0$. By letting $\psi(t)\to 0$, it degenerates to $V_0 = \{0\} \subset \psi^{-1}(0)$.\\ Notice that, for $\psi(t)\neq 0$, Eq. \eqref{eq thimble WITH NO CRITICAL VALUE in local coordinates} can be used to endow the vanishing cycle over $\psi(t)$ of the thimble with a canonical parametrization
\[ \theta\mapsto  \big(\sqrt{\psi(t)} \cos(\theta),\sqrt{\psi(t)}\sin(\theta)\big).\]
Let $\Delta$ be the family of Lagrangian spheres
\[\Delta := \bigcup_{z \in \C}V_z\cup x_0.\]
We conclude with some notation which will be used throughout the paper.
\begin{notation}
     The space of smooth (resp. continuous) sections of a given map $\pi: E \to B$ will be denoted with $\Gamma(E)$ (resp $\Gamma_0(E)$).
\end{notation}

\subsection{Smooth and continuous sections of Lefschetz fibrations over the disk with a single critical value}
\subsubsection{Lefschetz fibrations over the disk with a single critical value}\hfill\\

\noindent Our construction of smooth and continuous sections for a Lefschetz fibration over the disk requires an explicit model for a Lefschetz fibration over the disk. The following construction taken from \cite{Sei03} will be used. We included some details we believe might be of independent interest to the reader.
\begin{prop}\label{prop seidel existence Lef fibration over disk}
    Let $V$ be an essential curve in $\Sigma$. Fix some radius $r_0>0$. There is a symplectic Lefschetz fibration $(E^V,\pi^V)$ over $D^V(r_0)$ (the disk of radius $r_0$) together with an symplectomorphism between the regular fiber over $r_0\in \partial D^V(r_0)$ with its induced symplectic form and a surface $\Sigma$ with appropriate volume form
    \[\phi^V:E_{r_0}^V\to \Sigma,\] 
    such that if $\rho^V$ is the symplectic monodromy around $\partial D^V(r_0)$, then ${\tau_V=\phi^V \rho^V(\phi^V)^{-1}}$ is a Dehn twist along $V$. The Lefschetz fibration $(E^V,\pi^V)$ has a single critical point $x_0$ over $0\in D^V(r_0)$.
\end{prop}
\begin{proof}[Idea of proof.]
    For all the details, see \cite[Prop 1.11]{Sei03}. The construction consists in gluing an appropriate local model $M\subset \C^2$ around the critical point $x_0$ together with the quadratic complex projection map $q: M\to \C$, to a chosen neighborhood $\mathcal{N}(V)$ of the vanishing cycle $V$ in the reference fiber $\Sigma$. Using Darboux coordinates, we can identify $\mathcal{N}(V)$ with $T^*_{\leq \lambda}(S^1)$, the subbundle of the cotangent bundle whose vectors have norm at most $\lambda>0$.\\
The model $M$ is defined via a diffeomorphism $\Phi$ fibered over $\C$ (see \cite[Eq. 1.13]{Sei03}):
\[ M := \Phi^{-1}\big(D\times (T^*_{\leq \lambda}(S^1)\setminus S^1)\big)\cup \Delta.\]
The diffeomorphism is given explicitly as
\begin{equation}\label{eq. Seidel diffeomorphism local model}
        \begin{split}
            \Phi : \C^2\setminus \Delta &\to \C \times (T^*(S^1)\setminus S^1)\\
            \textbf{x} &\mapsto (q(\textbf{x}),\cos(\alpha/2)u-\sin(\alpha/2)\|u\|v+\sin(\alpha/2)\frac{u}{\|u\|})
        \end{split}
    \end{equation}
where $se^{i\alpha}=q(\textbf{x})$, $u:=-\Im(e^{-i\alpha/2}\textbf{x})\|\Re(e^{-i\alpha/2}\textbf{x})\|$ and $v:=\frac{\Re(e^{-i\alpha/2}\textbf{x})}{\|\Re(e^{-i\alpha/2}\textbf{x})\|}$.\\
Finally
\[ E^V := M \ \bigcup_{\mathlarger{\sim}} \ D^V(r_0)\times (\Sigma \setminus \mathcal{N}(V))\]
where the identification uses $\Phi$.
 \end{proof}
 \begin{rem}\label{rem extension Phi to whole fiber}
     We recall here the following observation of Seidel (\cite[page 14]{Sei03}): if we restrict $\Phi$ to a fiber over $s>0$, then it extends to a diffeomorphism \[q^{-1}(s)\xrightarrow{\cong} T_{\leq \lambda}^*(S^1).\] 
 \end{rem}
By restricting the radius of the disk if necessary, we can assume that the holomorphic neighborhood $\mathcal{U} \cong\C^2$ around the critical point $x_0$ is such that $\pi^V(\mathcal{U})=D^V(r_0)$.\\
Let $\ast_{D}:=r_0 \in D^V(r_0)$ be the chosen base-point in the base disk. Given $(X^V,\pi^V)$ as above, let $\mathcal{N}(V)$ be an open neighborhood of $V$ in $\Sigma$ arising from Prop. \ref{prop seidel existence Lef fibration over disk}. Let $\ast_{\Sigma}\in \Sigma \setminus \mathcal{N}(V)$. Fix, via the isomorphism $\phi^V$, an identification of $F_{\ast_D}\cong \Sigma$ and let $\ast_{E}=(\phi^V)^{-1}(\ast_{\Sigma})\in E^V$.\\ In the rest of this work, we will drop the reference to a specific radius $r_0$ as it will not affect the results.
\subsubsection{Sections of mapping torii}\hfill\\

\noindent Given a diffeomorphism $f:\Sigma\to \Sigma$, the mapping torus of $f$ is the following space:
     \[\Sigma_f := \bigslant{\Sigma\times [0,1]}{(x,1)\sim (f(x),0)}.\]
 Mapping torii arise naturally as the boundary of Lefschetz fibrations over a disk $D$. Thanks to \cite[Rem 1.4]{Sei03}, on a chosen collar neighborhood of $\partial D$ our fibration can be taken to be flat. We have then an induced canonical isomorphism
\begin{equation}\label{eq canonical identification boundary EV with mapping torus.}
    \restr{E^V}{\partial D^V}\cong \Sigma_{\tau_V}.
\end{equation}
\begin{rem}\label{rem flat boundary mapping torus}
Our chosen convention for the mapping torus of $f$ implies the following:
\begin{itemize}
    \item[(i)] Parallel transport along the base in the positive direction coincides with the map $f$
    \item[(ii)] $\pi_1(\Sigma_f,\ast_{\Sigma})\cong \langle \pi_1(\Sigma,\ast_{\Sigma}),t\mid ta t^{-1}=f_*^{-1}(a) \ \forall a \in \pi_1(\Sigma,\ast_{\Sigma})\rangle$, i.e. $\Z$ acts via the inverse of $f$.
\end{itemize}
\end{rem}
Choose appropriate coordinates $\{(u,v)\in \C\times [0,1]\mid \|u\|=1\}$ for a neighborhood of the zero section of $T^*S^1$ containing the support of $\tau_V$ and identified with $\mathcal{N}(V)$ in $\Sigma$. By construction $V=\{(u,1/2) \mid u \in S^1\}$ and 
\begin{align*}
    \tau_V : S^1\times [0,1] &\to S^1\times [0,1]\\
    (u,v)&\mapsto (ue^{2\pi iv},v).
\end{align*}
Let $\widetilde{V}=\{(u,1)\mid u \in S^1\}$. For convenience we assume that this loop is based at $\ast_{\Sigma}$. We can define the sections $s_{\widetilde{V}},s_{V}$ of $\Sigma_{\tau_V}$ as follows:

\begin{minipage}[c]{0.5\textwidth}
\begin{align*}
    s_{V} : S^1 &\to \mathcal{N}(V)_{\tau_V}\subset \Sigma_{\tau_V}\\
   \theta &\mapsto \big[(\theta, (e^{ i\theta/2},1/2))\big]
\end{align*}
\end{minipage}
\begin{minipage}[c]{0.5\textwidth}
\begin{align*}
    s_{\widetilde{V}} : S^1 &\to \mathcal{N}(V)_{\tau_V}\subset \Sigma_{\tau_V}\\
    \theta &\mapsto \big[(\theta, (e^{i\theta},1))\big].
\end{align*}
\end{minipage}\\

These sections are well-defined in $\Sigma_{\tau_V}$: 
\begin{align*}
    &s_{V}(2\pi) = \big[(2\pi, (e^{\pi i },1/2))\big] =\big[(0,\tau_V(e^{\pi i},1/2))\big] =[0,(1,1/2)] = s_{V}(0) \\
    &s_{\widetilde{V}}(2\pi) =\big[(2\pi, (e^{2\pi i},1))\big]=\big[(0,\tau_V(e^{2\pi i},1))\big]=[0,(1,1)]=s_{\widetilde{V}}(0)
\end{align*}
Interpolating between the two construction yields a free homotopy \textit{of sections} 
\[H_V :S^1\times [0,1] \to \Sigma_{\tau_V}\] 
between $s_{\widetilde{V}}$ and $s_{V}$ (See Fig. \ref{fig:freehomotopysectionsMapTorus}).\\
Similarly, we can define:\\
\begin{minipage}[c]{0.5\textwidth}
\begin{align*}
    s_{V^{m}} : S^1 &\to \Sigma_{\tau_V}\\
    \theta &\mapsto [\theta, (e^{i\tfrac{2m-1}{2}\theta},1/2)]
\end{align*}
\end{minipage}
\begin{minipage}[c]{0.5\textwidth}
\begin{align*}
    s_{\widetilde{V}^{m}} : S^1 &\to \Sigma_{\tau_V}\\
    \theta &\mapsto [\theta, (e^{im\theta},1)].
\end{align*}
\end{minipage}\\

Let us denote with ``$\star$'' the operation of loop concatenation. Given a loop $\gamma$ and integer $m$, if $m\geq 0$ then $\gamma^m$ denotes the $m$-fold concatenation
\[\gamma^m:= \underbrace{\gamma \star \cdots \star \gamma}_{m-\text{ times}}.\]
If $m<0$ we define \[ \gamma^m := \overline{\gamma}^{|m|}.\]
Since $s_{\widetilde{V}}$ is based at $\ast_{\Sigma}$, the section $s_{\widetilde{V}^{m}}$ represents the class 
\[([\widetilde{V}^m],1)\in \pi_1(\Sigma_{\tau},\ast_{\Sigma})\cong \pi_1(\Sigma,\ast_{\Sigma})\rtimes_{{\tau_V}_*} \Z,\] as it wraps around the (based) curve $\tilde{V}$ $m$-times. We can generalize the isotopy of Fig. \ref{fig:freehomotopysectionsMapTorus} to get a family of sections connecting $s_{V^{m}}$ to $s_{\widetilde{V}^{m}}$.
\begin{rem}
    It is well known that the effect of the Dehn twist $\tau_V$ is independent on the choice of parametrization (hence orientation) of $V$. Our convention is reflected in the picture above where $V$ is oriented counterclockwise. Changing orientation does not change the isotopies, but changes the label we use for the sections.
\end{rem}
\begin{figure}[H]
    \centering
    \includegraphics[width=.8\linewidth]{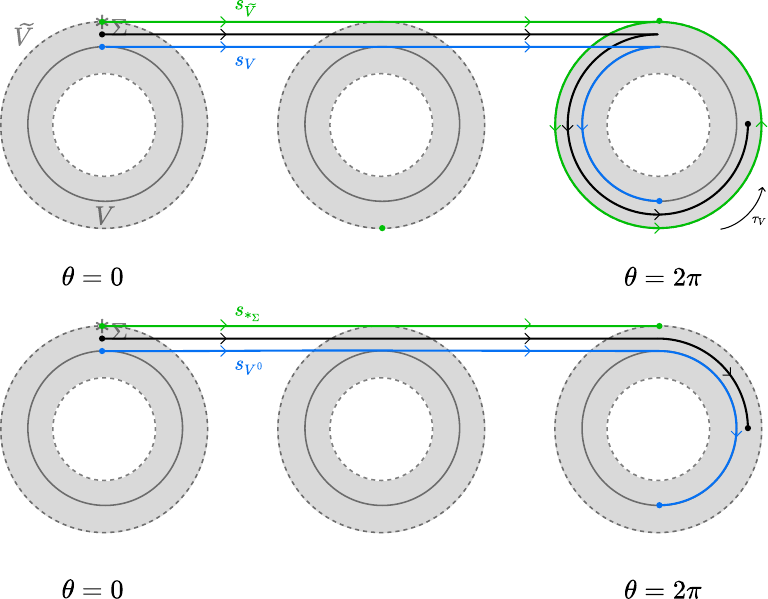}
    \caption{Family of sections of $\Sigma_{\tau_V}$ realizing an homotopy-through-sections starting at $s_{V}$ (resp. $s_{V^{0}}$) and ending at $s_{\widetilde{V}}$ (resp. $s_{\ast_{\Sigma}}$).}
    \label{fig:freehomotopysectionsMapTorus}
\end{figure}

 The following is a simple but useful observation
\begin{lemma} \label{lem existence of sections of mapping torus for any homotopy class of fibre loops}
    If $\ast_{\Sigma}\in \mathrm{Fix}(f)$, then for any class $[\alpha] \in \pi_1(\Sigma,\ast_{\Sigma})$ there is a smooth section $s_\alpha \in \Gamma(\Sigma_f)$ realizing the class 
    \[ ([\alpha],1) \in \pi_1(\Sigma_f,\ast_{\Sigma})\rtimes_{f_*}\Z.\] 
\end{lemma}
\begin{proof}
    Any section of the mapping torus of $f$ can be seen as a twisted loop, i.e. an element of the twisted loop space $\Omega_f(\Sigma)$:
  \[ \Omega_f(\Sigma) := \{\rho :[0,1]\to \Sigma \mid \rho(1)=f\rho(0)\}.\]
    The lemma is then an immediate consequence of the equivalence \[\Gamma(\Sigma_f) = \Omega_f(\Sigma).\]
\end{proof}
Given two sections $s_0,s_1 \in \Gamma(\Sigma_f)$, we use the notation $s_0\sim_{S^1} s_1$ to denote that $s_0$ and $s_1$ are homotopic through sections.
\begin{lemma}\label{lem section homotopic through sections iff conjugate by fiber element} Given two based sections $s_0,s_1 \in \Gamma(\Sigma_f)$ of the mapping torus $\Sigma_f$ representing in $\pi_1(\Sigma_f,\ast_{\Sigma})$ the elements $([\beta],1)$ and $([\alpha],1)$ respectively, then TFAE:
\begin{enumerate}
    \item[(a)]  $s_0\sim_{S^1} s_1$
    \item[(b)]  $[\alpha]= [\zeta]\cdot [\beta]\cdot f^{-1}_*([\overline{\zeta}])$ in $\pi_1(\Sigma,\ast_{\Sigma})$.
\end{enumerate}
\end{lemma}
\begin{proof}
    Recall that $\Gamma(\Sigma_f) = \Omega_f(\Sigma)$. An homotopy through sections corresponds to a path in $\Omega_f(\Sigma)$. Two elements are in the same connected component of $\Omega_f(\Sigma)$ if and only if exists $[\zeta]\in \pi_1(\Sigma,\ast_{\Sigma})$ such that \[([\alpha],1) = ([\zeta],0)\cdot ([\beta],1)\cdot ([\overline{\zeta}],0).\] This is a consequence of the long exact sequence of homotopy groups for the fibration $\Omega(\Sigma)\to \Omega_f(\Sigma)\to \Sigma$. Spelling out the operation group operation completely we get
    \[ [\alpha]= [\zeta]\cdot [\beta]\cdot f^{-1}_*([\overline{\zeta}]), \]
    concluding the proof.
\end{proof}
\subsubsection{Sections of Lefschetz fibrations over the disk with a single critical value}\hfill\\

\noindent Fix the identification 
\begin{equation}\label{eq identification boundary fundamental group with semidirect product}
    \pi_1\big(\restr{E^V}{\partial D^V},\ast_E\big)\cong \pi_1\big(\Sigma,\ast_{\Sigma}\big) \rtimes_{{\tau_V}_*} \Z
\end{equation}
described in Rem. \ref{rem flat boundary mapping torus} together with the horizontal section of $\Sigma_{\tau_V}$ swept by $\ast_{\Sigma}$.

\begin{theorem}\label{thm existence section LF over disk with single thimble} Let $(E^V,\pi^V)$ be as above. For any $m\in \Z$ and any $[\alpha]\in \pi_1(\Sigma,\ast_{\Sigma})$ satisfying \[[\alpha]= [\zeta]\cdot [\widetilde{V}^m]\cdot {\tau_V}^{-1}_*([\overline{\zeta}])\] for some loop $\zeta$, there exists a continuous section \[\sigma_{\alpha}^V : D^V\to E^V\]
whose restriction to $\partial D^V$ satisfies the following property:
\[ [ \restr{\sigma_{\alpha}^V}{\partial D^V}]=([\alpha],1) \in \pi_1(\Sigma,\ast_{\Sigma}) \rtimes_{{\tau_V}_*} \Z\]
under the identification from \eqref{eq canonical identification boundary EV with mapping torus.}
\end{theorem}
\begin{proof} 
   Let $s_\alpha\in \Gamma(\Sigma_{\tau_V})$ be a section representing the element $([\alpha],1)$ (Lemma \ref{lem existence of sections of mapping torus for any homotopy class of fibre loops}). By chaining together the homotopies through sections arising from Lemma \ref{lem section homotopic through sections iff conjugate by fiber element} and Fig. \ref{fig:freehomotopysectionsMapTorus} we have: \[ s_\alpha\sim_{S^1} s_{\widetilde{V}^m}\sim_{S^1} s_{V^m}.\]
 Let $\Sigma_{\tau_V}\times [0,1]\to A$ be the thickened mapping torus over the annulus $A$ and let $h : A \to \Sigma_{\tau_V}$ be the homotopy of sections between the sections $s_\alpha$ and $s_{V^m}$. After possibly extending our fibration $E^V\to D^V$ over $A$ by identify the the mappng torus over $\partial D^V$ with the one over the interior boundary component of $A$ and extending $s_\alpha$ over $A$ by using $h$, we can assume that $s_a= s_{V^m}$ on $\partial D^V$.\\
   Parametrize $D^V$ using the complex exponential \[[0,1]\times [0,2\pi]\ni (r,\theta)\mapsto re^{i\theta}.\] Consider the following map, whose image lands in an holomorphic neighborhood of the critical point
\begin{equation}\label{eq section over disk tracing mV}\begin{split}
\sigma_{s}^V : D^V &\to E^V\\
    (re^{i\theta}) &\mapsto \bigg(\sqrt{r}e^{i\theta/2} \cos{\big(\tfrac{2m-1}{2}\theta\big)},\sqrt{r}e^{i\theta/2}\sin{\big(\tfrac{2m-1}{2}\theta\big)}\bigg).
\end{split}
\end{equation} 
The formula is well-defined, as for $re^0=re^{i2\pi}$, the value of $\sigma_{s}^V$ is the same. Continuity is clear from the formula since when approaching $0$, $\sigma_s^V$ limits uniformly to $(0,0)\in \C^2$. Lastly, $q\circ \sigma_{s}^V(re^{i\theta})=re^{i\theta}$ can be verified by applying the local coordinate expression for the projection map $\pi$.\\ 
It remains to show that $\sigma_{s}^V$ defines the correct homotopy class under the identification of $\restr{E^V}{\partial D^V}$ with $\Sigma_{\tau_V}$. First notice that $\restr{\sigma_s^V}{\partial D}$ coincides with the section $s_{V^m}$. By then enlarging the fibration via gluing a thickened mapping torus $\Sigma_{\tau_V}\times [0,1]\to S^1\times [0,1]$ and extending the section over the annulus via the isotopy $H_V^{(m)}$ we obtained that $\restr{\sigma_s^V}{\partial D}=s_{\widetilde{V}^m}$, which clearly represents the element $([\widetilde{V}^m],1) \in \pi_1(\Sigma,\ast_{\Sigma}) \rtimes_{{\tau_V}_*} \Z$.
\end{proof}

\begin{rem}\label{rem smooth and continuous sections over the disk}
The sections $\sigma_m^V$ build in Thm. \ref{thm existence section LF over disk with single thimble} are continuous but not smooth as they pass through the critical point. Seidel, in \cite[Lemma 2.16]{Sei03}, shows that for a small enough base radius $r$, pseudo-holomorphic sections $w$ of the Lefschetz fibration ${E^V\to D^V}$ satisfying the (Lagrangian) boundary condition
\[ \restr{w}{\partial D}\subset \Delta\] are of the form:
\begin{align*}
    w_{\textbf{a}} : D^V &\to \C^2\subset E^V\\
    z &\mapsto \dfrac{\textbf{a}}{\sqrt{r}}z + \sqrt{r}\overline{\textbf{a}}
\end{align*}
where $\textbf{a}\in \C^2$ satisfies $\|\Re(\textbf{a})\|^2=\|\Im(\textbf{a})\|^2=1/4$ and $\langle \Re(\textbf{a}), \Im(\textbf{a})\rangle=0$. The set of suitable vectors $\textbf{a}$ can be identified with the sphere bundle $S(T^*S^1)={S^1}_+\bigsqcup {S^1}_-$.\\ Following our convention for the orientation on $V$, we see that for $\textbf{a}\in {S^1}_+$ (resp. $\textbf{a}\in {S^1}_-$)  \[ \restr{w_{\textbf{a}}}{\partial D} \simeq s_{V} \ \ \ (\text{resp. } \restr{w_{\textbf{a}}}{\partial D} \simeq s_{V^{-1}}).\]
\end{rem}

The following result can be interpreted a sort of converse for Rem. \ref{rem smooth and continuous sections over the disk}. Using Lemma \ref{lem existence of sections of mapping torus for any homotopy class of fibre loops} and the identification from \eqref{eq canonical identification boundary EV with mapping torus.} we can identify a section of the boundary of $E^V$ with an element  $([\alpha],1)\in \pi_1(\Sigma,\ast_{\Sigma})\rtimes_{{\tau_V}_*}\Z$.

\begin{theorem}\label{theorem smooth and continuous sections criteria for sections over the disk}
    Let $(E^V,\pi^V)$ be as above and let $s\in \Gamma(\restr{E^V}{\partial D})$ be a section representing the element $([\alpha],1)\in \pi_1(\Sigma,\ast_{\Sigma})\rtimes_{{\tau_V}_*}\Z$.\\
    There exists a section $\sigma\in \Gamma_0(E^V)$ over $D$ extending $s$ and going through the cri\-ti\-cal point of $\pi^V$ if and only if there exists a based loop $\zeta : S^1\to \Sigma$ such that  
    \[\alpha=[\zeta]\cdot[\widetilde{V}^m]\cdot {\tau_V}^{-1}_*([\overline{\zeta}]) \in \pi_1(\Sigma,\ast_{\Sigma})\]
 for some $m\in \Z$. Such section can be taken to avoid the critical point (and therefore smoothed) if and only if $m=0,1$.
\end{theorem}
\begin{proof}
 We first prove the statement concerning smoothable sections.
    \begin{itemize}
        \item[$(\Leftarrow)$] This is immediate from Rem. \ref{rem smooth and continuous sections over the disk} together with the isotopies represented in Fig. \ref{fig:freehomotopysectionsMapTorus}.
    \item[$(\Rightarrow)$] Let $\sigma$ be a continuous section of $E^V\to D^V$ such that $\sigma(1)=\ast_{\Sigma}\in E^V_1$, $\sigma(0)\neq x_0$ and let $s=\restr{\sigma}{\partial D}$.\\ We first prove that under our assumption of avoiding the critical point, $\sigma$ can be smoothed relative to $s$. Since $\sigma(0)\neq x_0$, on a sufficiently small disk $D'\subset D$ around the critical value, \[ \mathrm{Im}(\sigma)\cap \restr{\Delta}{D'}=\varnothing.\]
    By the trivialization provided in Eq. \eqref{eq. Seidel diffeomorphism local model}, 
    \[ \restr{E^V}{D'}\setminus \restr{\Delta}{D'} \cong D'\times \Sigma\setminus V,\]
    hence $\sigma$ is a section of an actual (trivial) fibration over $D'$. Therefore we can smooth $\restr{\sigma}{D'}$ and then smooth $\sigma$ in the complement of $D'$ relative to $s$ and $\restr{\sigma}{\partial D'}$, obtaining a smooth section $\sigma'$ with the same boundary condition.\\
   Assume now that $\sigma$ is smooth. By Prop. \ref{prop seidel existence Lef fibration over disk}, outside a neighborhood of the critical point, the fibration $E^V\to D^V$ can be taken to be flat. Hence there is a well-defined family of basepoints, one for each fiber over the radius $r(t)=t:[0,1]\to D$, obtained by parallel transport of $\ast_E\in F_{\ast_D}$ along any path in the base.\\
   Assume $V$ is non-separating: by using the same reasoning as in the smoothing part of the proof, if we choose a small enough disk around the critical value then the section will avoid $\Delta$ entirely over such disk $D'$. Therefore, thanks to Eq. \eqref{eq. Seidel diffeomorphism local model}, the restriction of $\sigma$ can be considered a section of the trivial fibration $D'\times \Sigma\setminus V$. The following observations are then immediate and the proof is left to the reader:
    \begin{itemize}
        \item[(a)] We can assume that $\restr{\sigma}{\partial D'}$ is a based loop and
        \item[(b)] $\restr{\sigma}{\partial D'}\sim_{S^1} s_{\ast_{\Sigma}}$. 
    \end{itemize}
    The restriction $\restr{\sigma}{D\setminus D'}$ defines an homotopy through sections between $\restr{\sigma}{\partial D'}$ and $\restr{\sigma}{\partial D}=s$, proving \[s \sim_{S^1} s_{\ast_{\Sigma}}.\] By Lemma \ref{lem section homotopic through sections iff conjugate by fiber element} we conclude.\\
    Assume now that $V$ is separating. If $\ast_{\Sigma}$ and $\sigma(0)$ belong to the same connected component of $E^V_0\setminus \{x_0\}$ the same proof in the non-separating case can be applied. If $\ast_{\Sigma}$ and $\sigma(0)$ belong to different connected components, by choosing a new base point $\ast'$ in the appropriate connected component we conclude as before that
    \[ \restr{\sigma}{\partial D'} \sim_{S^1} s_{\ast'}.\]
     The homotopy through section represented in Fig. \ref{fig:isotopyVsepdiffcc} then shows \[\restr{\sigma}{\partial D'}\sim_{S^1} s_V.\]
    As in the previous cases we then conclude that $s \sim_{S^1} s_{\widetilde{V}}$.
    \end{itemize}
We now prove the general statement.
    \begin{itemize}
        \item[$(\Leftarrow)$] This is immediate from Thm. \ref{thm existence section LF over disk with single thimble}. 
        \item[$(\Rightarrow)$] Let $\sigma\in \Gamma(E^V)$ such that $\sigma(0)=x_0$, the critical point. Since $\sigma$ is continuous and $x_0\in M$, there exists $D''\subset D$ such that $\restr{\sigma}{D''}\subset M\cap \pi^{-1}(D'')$. Thanks to Prop. \ref{prop seidel existence Lef fibration over disk} and Rem. \ref{rem extension Phi to whole fiber}, over each $z\in D''$, $M\cap \pi^{-1}(\{z\})\cong T^*_{\leq \lambda}(S^1)$, hence \[ M\cap \pi^{-1}(\partial D'') \cong T_{\leq \lambda}(S^1)_{\tau_V},\] the mapping torus of the restriction of the Dehn twist about $V$ to a tubular neighborhood of $V$.\\
    The following change of coordinates shows that $T_{\leq \lambda}(V)_{\tau_V}$ is isomorphic to the trivial bundle $\R/\Z\times \R/\Z\times [0,1]\to \R/\Z$ 
    \begin{align*}
         \bigslant{T_{\leq 1}^*(S^1)\times [0,1]}{(\tau_V(\theta,s),0)\sim ((\theta,s),1)}&\to \R/\Z\times \R/\Z\times [0,1]\\
         [(\theta,s,t)] &\mapsto ([\theta+st],[t],s).
    \end{align*}
     The identification sends the section swept by $(0,1)\in \partial T_{\leq 1}(V)$ to the loop representing the element $(0,1)\in \Z^2\cong \pi_1(S^1\times S^1\times [0,1])$.\\ 
     Since $\sigma_{\restr \partial{D''}}$ defines a loop in $T_{\leq \lambda}(V)_{\tau_V}$, it has to be of the form $(m,1)\in \Z^2$, for an appropriate base point on $\R/\Z$.\\
     By contracting along the $[0,1]$ factor, we obtain a homotopy through section 
     \[ \sigma_{\restr \partial{D''}} \sim s_{V^m}.\]
     Using the appropriate homotopies from Fig. \ref{fig:freehomotopysectionsMapTorus} and Lemma \ref{lem section homotopic through sections iff conjugate by fiber element} we conclude.
    \end{itemize}
\end{proof}

\begin{figure}[h]
    \centering
    \includegraphics[width=0.8\linewidth]{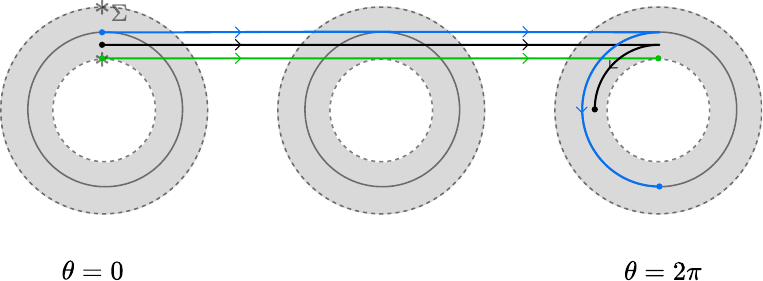}
    \caption{The homotopy of sections between $\restr{\sigma}{\partial D'}$ and $s_{\widetilde{V}}$ when $\sigma(0)$ and $\ast_{\Sigma}$  belong to different connected components of $\mathcal{N}(V)\setminus V$.}
    \label{fig:isotopyVsepdiffcc}
\end{figure}

\subsection{Smooth and continuous sections for general Lefschetz fibrations over the disk}
\subsubsection{Lefschetz fibrations over the disk with several critical values.}\hfill\\
Let $\Sigma$ be an oriented surface and let $\tau\in \mathrm{Aut}(\Sigma)$ satisfying 
\begin{equation}\label{eq positive factorisation monodromy}
    \tau=\tau_{V_k}\circ \cdots \circ \tau_{V_1},
\end{equation} 
for a family of embedded curves $V_1,\dots , V_k\subset \Sigma$. We can use this data to build a Lefschetz fibration $\pi: E\to D$ with $k$ critical values and the curves $V_i$'s as vanishing cycles.\\
Mark points $\ast_D\in \partial D$ and $z_1,\dots,z_k \in \mathrm{int}(D)$; choose $k$ embedded paths ${\psi_i:[0,1]\to D}$ such that 
\begin{itemize}
    \item $\forall i$, $\psi_i(0)=\ast_D$, $\psi_i\big((0,1]\big)\subset \mathrm{Int}(D)$ and $\psi_i(1)=z_i$
    \item the paths are all disjoint from each other except at $\ast_D$.
\end{itemize}
In the given disk $D$ we can find a bouquet of $k$ disks $D_i\ni z_i$, $\bigvee_{\ast_D} D_i$,
disjoint everywhere but on the boundary point $\ast_D\in \partial D$. Denote with $\gamma_i$ the based loop at $\ast_D\in D$ parameterizing the boundary of $D_i$ in the positive direction.
\begin{defi}
    A \textit{geometric} or \textit{Artin basis} for $\pi_1(D\setminus \{z_1,\dots,z_k\},\ast_D)$ is such an ordered  $k$-uple $([\gamma_1],\dots ,[\gamma_k])$.\\
    A geometric basis for $D\setminus \{z_1,\dots,z_k\}$ is an ordered $k$-uple of simple smooth representatives for a geometric basis for $\pi_1(D\setminus \{z_1,\dots,z_k\},\ast_D)$
\end{defi}
For each $i\in \{ 1,\dots,k\}$, let $(E^{V_i},D^{V_i})$ be the Lefschetz fibration with fiber $\Sigma$ and single vanishing cycle $V_i$ built via Prop. \ref{prop seidel existence Lef fibration over disk}. For each index $i$, we identify the disks $D_i$ and $D^{V_i}$ and  $\ast_D$ with $1\in \partial D^{V_i}$.\\
Since $D\setminus \bigvee_{\ast_D} D_i$ is contractible, we can extend the fibration on it by simply pulling back the restriction of the fibration over the bouquet of boundaries $\bigvee_{\ast_D}\gamma_i$.\\
This construction depends on the choice of factorization \eqref{eq positive factorisation monodromy} of $\tau$. An e\-qui\-va\-lent positive factorization of the monodromy will transform the geometric basis by the appropriate sequence of Hurwitz moves and yield an isomorphic Lefschetz fibration.
\begin{rem}\label{rem sectorial structure and ordered k-uple of vanishing cycles}
    We can always assume that, up to isomorphism, any Lefschetz fibration $\pi:E\to D$ over the disk with fiber $\Sigma$ arises from this construction. Let $\{z_1,\dots,z_k\}\subset \mathrm{Int}(D)$ be the set of critical values of $\pi$. Choose a geometric basis $\{\gamma_i\}_{i=1}^k$ for $D\setminus \{z_1,\dots,z_k\}$,  then the restriction of $\pi$ to each $D_i$ satisfies
    \[ \restr{E}{D_i}\cong E^{V_i},\]
    where $V_i$ is the vanishing cycle arising from the vanishing path $\psi_i$. On $D\setminus \bigvee_{\ast_D} D_i$, the fibration can be trivialized, providing the required isomorphism.\\
    By successive applications of the identification provided by \eqref{eq canonical identification boundary EV with mapping torus.}, we have a preferred identification
    \[ \restr{E}{\partial D}\cong \Sigma_{\tau}.\]
\end{rem}

\begin{lemma}\label{lem section over D same as section over bounquet of disks}
    Let $\pi: E\to D$ be a Lefschetz fibration as above. Up to homotopy through sections, we have the following identifications
\begin{equation*}
    \begin{split}
        \Gamma(E)&\cong\Gamma\big(\restr{E}{\bigvee_{\ast_D}D_i}\big)\\
        \Gamma_0(E)&\cong\Gamma_0\big (\restr{E}{\bigvee_{\ast_D}D_i}\big)
    \end{split}
\end{equation*}
\end{lemma}
\begin{proof}
Let $\pi: E\to D$ be the given Lefschetz fibration, choose a geometric basis for $D\setminus \{z_1,\dots,z_k\}$. By restricting $D$ to the bouquet of disks, we immediately obtain maps 
\begin{equation*}
    \begin{split}
        \Gamma(E)&\to\Gamma\big(\restr{E}{\bigvee_{\ast_D}D_i}\big)\\
        \Gamma_0(E)&\to\Gamma_0\big (\restr{E}{\bigvee_{\ast_D}D_i}\big)
    \end{split}
\end{equation*}
    The inverse maps are obtained by trivializing $E$ over $D\setminus \bigvee_{\ast_D}D_i$ and using such trivialization to extend the section to the whole $D$. This also shows that up to homotopy of sections, the two maps are inverse of each other.
\end{proof}

\subsubsection{Sections of concantenated mapping torii}\hfill\\
In order to characterize the loops that can arise as boundary of sections in $\Gamma\big(\restr{E}{\bigvee_{\ast_D}D_i}\big)$ or $\Gamma_0\big(\restr{E}{\bigvee_{\ast_D}D_i}\big)$, we need a suitable description of $\pi_1(\Sigma_{\tau},\ast_{\Sigma})$.\\ 
Let $F_k$ be the free group generated by the set $\{[\gamma_i]\}_{i=1}^k$. Let
\begin{align*}
    \rho: F_k&\to \mathrm{Aut}\big(\pi_1(\Sigma,\ast_{\Sigma})\big)\\
     \rho([\gamma_i])([\alpha])&={\tau_{V_i}}_*^{-1}([\alpha]),
\end{align*}
using this representation we can define the semidirect product $\pi_1(\Sigma,\ast_{\Sigma})\rtimes_\rho F_k$. We will use the symbol ``$\bullet$'' to denote its group operation. An inductive application of Seifert-Van Kampen theorem reveals that
\[ \pi_1(\restr{E}{\bigvee_{\ast_D}\gamma_i},\ast_E)\cong \pi_1(\Sigma,\ast_{\Sigma})\rtimes_\rho F_k.\]
Choose a geometric basis $\vec{\gamma}:=(\gamma_1,\dots , \gamma_k)$ of $D\setminus \{z_1,\dots,z_k\}$. We define the map
\[\partial D \hookrightarrow D \setminus \{z_1,\dots,z_k\} \simeq \bigvee_{\ast_D}^k \gamma_i.\]
On total spaces this gives us the inclusion 
\[\Sigma_{\tau} \cong \restr{E}{\partial D} \hookrightarrow \restr{E}{\bigvee_{\ast_D}\gamma_i}\]
which induces the following map at the level of fundamental groups: 
\begin{equation}\label{eq definition inclusion mapping torus of composition monodromy into fiber sum of mapping torus}
    \begin{split}
        \iota_{\vec{\gamma}} : \pi_1(\Sigma)\rtimes_{\tau_*} \mathbb{Z} &\to \pi_1(\Sigma,\ast_{\Sigma})\rtimes_\rho F_k\\
        ([\alpha],0) &\mapsto ([\alpha],0)\\
        (0,1) &\mapsto (0,[\gamma_1])\bullet \cdots \bullet(0,[\gamma_k]).
    \end{split} 
\end{equation}
The following lemma describes how to \textit{concatenate} homotopy classes of sections of $\Sigma_{\tau_{V_i}}$ to obtain a well-defined homotopy class of $\pi_1(\Sigma_{\tau},\ast_{\Sigma})$. 
\begin{lemma}\label{lemma convenient identification of fundamental group mapping torus as subgroup of bouquet of mapping torii}
    The map $\iota_{\vec{\gamma}}$ from \eqref{eq definition inclusion mapping torus of composition monodromy into fiber sum of mapping torus} is well defined and injective. Given a sequence of sections $s_i:S^1\to \Sigma_{\tau_{V_i}}$, representing homotopy classes \[([\alpha_i],1) \in \pi_1(\Sigma_{\tau_{V_i}},\ast_{\Sigma}),\] there is a unique class  $([\alpha],1)\in \pi_1(\Sigma_{\tau},\ast_{\Sigma})$ 
    which represents the homotopy class of the concatenation of the $s_i$ as a loop in $\Sigma_{\tau}$ and whose representative $\alpha$ must satisfy
\begin{equation*}
    \alpha\sim \alpha_1\star \tau_{V_1}^{-1}\alpha_2\star \cdots \star (\tau_{V_{k-1}}\circ \cdots \circ \tau_{V_1})^{-1}\alpha_k \text{ in } \Sigma.
\end{equation*}
\end{lemma}
\begin{proof}
To prove well-definedness, the only relation to check is \[(0,1)([\alpha],0)(0,-1) = \big(\tau^{-1}_*([\alpha]),0\big).\] The map $\iota_{\vec{\gamma}}$ maps the LHS of the previous equation to
\[ \big(0,[\gamma_1] \cdots [\gamma_k]\big)\bullet\big([\alpha],0\big)\bullet\big(0,([\gamma_1] \cdots [\gamma_k])^{-1}\big)=\big({\tau_{V_1}^{-1}}_*\cdots \ {\tau_{V_k}^{-1}}_*([\alpha]),0\big)=\big(\tau_*^{-1}([\alpha]),0\big).\]
To show injectivity, every element of $\pi_1(\Sigma_{\tau},\ast_{\Sigma})$ can be written as $([\alpha],\,n)$ with $[\alpha]\in\pi_1(\Sigma,\ast_{\Sigma})$ and $n\in\mathbb{Z}$. Assume 
\[ \iota_{\vec{\gamma}}([\alpha],n)=([\ast_{\Sigma}],e),\]
where $e\in F_k$ is the trivial letter. By definition $[\alpha]=[\ast_{\Sigma}] \in\pi_1(\Sigma,\ast_{\Sigma})$ and $([\gamma_1] \cdots [\gamma_k])^n=e \in F_k$. Since the word $[\gamma_1] \cdots [\gamma_k]$ is clearly reduced in $F_k$, it has to have infinite order in it. Therefore $n=0$, proving injectiviness.\\
Let $\alpha:= \alpha_1\star \tau_{V_1}^{-1}\alpha_2\star \cdots \star (\tau_{V_{k-1}}\circ \cdots \circ \tau_{V_1})^{-1}\alpha_k$, given the sequence of sections $s_i$, by applying $\iota_{\vec{\gamma}}$ we have:
\begin{align*}
    \iota_{\vec{\gamma}}\big([s_1\star \cdots \star s_k]\big)&= \big([\alpha_1],[\gamma_1]\big)\bullet \big([\alpha_2],[\gamma_2]\big)\bullet \cdots \bullet \big([\alpha_k],[\gamma_k]\big)\\
    &=\big(\alpha_1\star \tau_{V_1}^{-1}\alpha_2\star \cdots \star (\tau_{V_{k-1}}\circ \cdots \circ \tau_{V_1})^{-1}\alpha_k,[\gamma_1] \cdots [\gamma_k]\big)\\
    &= \iota_{\vec{\gamma}}([\alpha],1)
\end{align*}
By injectiviness of $\iota_{\vec{\gamma}}$ we conclude.
\end{proof}
\subsubsection{Sections of Lefschetz fibrations over the disk with several critical values}\hfill\\
\begin{theorem}\label{theorem criterion existence section LF over disk with several thimbles}
Let  $\pi: E\to D$ be a Lefschetz fibration as above. Let $\Sigma_{\tau}$ be the mapping torus of $\tau \in \mathrm{Aut}(\Sigma)$ identified with $\restr{E}{\partial D}$. Let $([\alpha],1)\in \pi_1(\Sigma_{\tau},\ast_{\Sigma})$ and let $s_{\alpha}$ be a section representative for it.\\ There is a continuous {section} $\sigma_{\alpha}$ of $E$ extending $s_{\alpha}$ if and only if the following condition is met:\\
    Choose a compatible factorization of the monodromy $\tau$:
    \[\tau = \tau_{V_k}\circ \cdots \circ \tau_{V_1},\]
    then, as a based loop in $\Sigma$, $\alpha$ can be decomposed as
    \[
    \alpha \sim \alpha_1 \star \tau_{V_1}^{-1}\alpha_2 \star \cdots \star (\tau_{V_{k-1}}\circ \cdots \circ \tau_{V_1})^{-1}\alpha_k,
    \]
    where each $\alpha_i$ is ``twisted-conjugated’’ to some multiple of a based representative $\widetilde{V_i}$ of the $i$-th vanishing cycle $V_i$. I.e. there exist an integer $m_i$ and a based loop $\zeta_i$ in $\Sigma$ such that
    \[
    \alpha_i \sim \zeta_i \star \widetilde{V_i}^{m_i} \star \tau_{V_i}^{-1}(\overline{\zeta_i}).
    \]
    Moreover, the section $\sigma_{\alpha}$ is smoothable if and only if $\forall i \ m_i \in \{0,1\}$.
\end{theorem}
\begin{proof}
A choice of a positive factorization
    \[\tau=\tau_{V_k}\circ \cdots \circ\tau_{V_1}\] is equivalent to a choice of vanishing paths and hence of a geometric basis $(\gamma_1,\dots,\gamma_k)$ for $D\setminus \{z_1,\dots,z_k\}$.
\begin{itemize}
    \item[$(\Leftarrow)$] Let $([\alpha],1)\in \pi_1(\Sigma,\ast_{\Sigma})\rtimes_{\tau_*} \mathbb{Z}$ satisfy 
    \[
    \alpha \sim \alpha_1 \star \tau_{V_1}^{-1}\alpha_2 \star \cdots \star (\tau_{V_{k-1}}\circ \cdots \circ \tau_{V_1})^{-1}\alpha_k,
    \]
    where each $\alpha_i$ satisfies the respective ``twisted conjugacy'' condition in the statement. Then by using $\iota_{\vec{\gamma}}$ we have 
                        \[\iota_{\vec{\gamma}}([\alpha],1)=([\alpha_1],[\gamma_1])\bullet ([\alpha_2],[\gamma_2])\bullet \cdots \bullet ([\alpha_k],[\gamma_k]).\]
    By Lemma \ref{lem section over D same as section over bounquet of disks}, to construct a section $\sigma_{\alpha}\in \Gamma_0(E)$ it is enough to define sections $\sigma_i\in \Gamma_0(\restr{E}{D_i})$ for each index $i$. Assuming
    \[[\restr{\sigma_i}{\partial D_i}]= ([\alpha_i],1)\in \pi_1(\Sigma_{\tau_{V_i}},\ast_{\Sigma}), \]
    then Lemma \ref{lemma convenient identification of fundamental group mapping torus as subgroup of bouquet of mapping torii} will ensure that their concatenation represents the correct homotopy class $([\alpha],1)\in \pi_1(\Sigma_{\tau},\ast_{\Sigma})$.
    The conditions required by Theorem \ref{thm existence section LF over disk with single thimble} to guarantee existence of such sections representing the correct homotopy classes are sa\-tis\-fied by assumption.\\
    If $m_i=0$ or $1$, we can apply the smooth characterisation from Theorem \ref{theorem smooth and continuous sections criteria for sections over the disk} and Lemma \ref{lem section over D same as section over bounquet of disks}.
    \item[$(\Rightarrow)$] Let $(\gamma_1,\dots,\gamma_k)$ be a geometric basis for $D\setminus \{z_1,\dots,z_k\}$. Let $\sigma_{\alpha} \in \Gamma_0(E)$ be a section whose restriction to $\partial D$ represents $([\alpha],1)\in \pi_1(\Sigma_{\tau},\ast_{\Sigma})$. By Lemma \ref{lem section over D same as section over bounquet of disks}, $\sigma_{\alpha}$ induces a family of sections $\{\sigma_i\in \Gamma_0(\restr{E}{D_i})\}$. Each $\sigma_i$ defines a based loop $s_i: S^1\to \restr{E}{\gamma_i}$ representing $([\alpha_i],1)\in \pi_1(\Sigma_{\tau_{V_i}},\ast_{\Sigma})$ for some loop $\alpha_i$ in $\Sigma$. We claim that such $\alpha_i$ satisfies the twisted conjugacy property.\\
    By Theorem \ref{theorem smooth and continuous sections criteria for sections over the disk} applied to $\sigma_i$, there is an integer $m_i$ and a based loop $\zeta_i$ in $\Sigma$ such that in $\pi_1(\Sigma,\ast_{\Sigma})$,
         \[[\alpha_i]=[\zeta_i]\cdot[\widetilde{V_i}^{m_i}]\cdot {\tau_{V_i}}^{-1}_*([\overline{\zeta_i}]).\]
    By Lemma \ref{lemma convenient identification of fundamental group mapping torus as subgroup of bouquet of mapping torii}, the homotopy class
    \[ [s_{\alpha}]=([\alpha],1)\]
    can be decomposed as 
    \[([\alpha_1],[\gamma_1])\bullet ([\alpha_2],[\gamma_2])\bullet \cdots \bullet ([\alpha_k],[\gamma_k]),\]
    showing that in $\pi_1(\Sigma,\ast_{\Sigma})$,
    \[ \alpha\sim \alpha_1\star \tau_{V_1}^{-1}\alpha_2\star \cdots \star (\tau_{V_{k-1}}\circ \cdots \circ \tau_{V_1})^{-1}\alpha_k.\]
    If $\sigma_{\alpha} \in \Gamma(E)$, i.e. it is smooth, then each $\sigma_i$ will be smooth and hence Theorem \ref{theorem smooth and continuous sections criteria for sections over the disk} will restrict $m_i$ to be either $0$ or $1$, concluding the proof.
  \end{itemize}
\end{proof}
\begin{rem}\label{rem theorem independent on choice of vanishing paths}
    Let $E\to D$ be Lefschetz fibration with fiber $\Sigma$ and monodromy $\tau\in \mathrm{Aut}(\Sigma)$. The property of a section $s_{\alpha}\in \Gamma_0(\restr{E}{\partial D})$ to be \textit{extendable} over $D$ only depends on $[\alpha]$ (and $\pi: E\to D$), not on the specific choice of positive factorisation of $\tau$. We can see that explicitly by considering the following example for $k=2$ that also serves as a sketch of proof for the invariance of Theorem \ref{theorem criterion existence section LF over disk with several thimbles} under Hurwitz moves.\\
   Choose a geometric basis for $D\setminus \{z_1,z_2\}$  such that $\tau=\tau_{V_2}\tau_{V_1}$ for some embedded curves $V_1,V_2$ in the regular fiber $\pi^{-1}(\ast_D)\cong \Sigma$. By applying the (first) Hurwitz move, we obtain an equivalent decomposition of the monodromy as
    \[ \tau=\tau_{V_2}\tau_{V_1}\tau_{V_2}^{-1}\tau_{V_2}=\tau_{\tau_{V_2}(V_1)}\tau_{V_2}\]
    and associated vanishing cycles $(V_1',V_2')=(V_2,\tau_{V_2}(V_1))$. By Theorem \ref{theorem criterion existence section LF over disk with several thimbles}, up to twisted conjugation, the only elements in $\pi_1(\Sigma_{\tau},\ast_\Sigma)$ arising as boundaries of sections are of the form
    \[ ({\widetilde{V'}_1}^{m_1},[{\gamma_1}'])\bullet ({\widetilde{V'}_2}^{m_2},[{\gamma_2}']).\] Since $[\gamma_1']=[\gamma_2]$, direct substitution shows
    \begin{align*}
        ([\widetilde{V}_1']^{m_1},[{\gamma_1}'])\bullet ([\widetilde{V}_2']^{m_2},[{\gamma_2}']) &=([\widetilde{V}_2]^{m_1},[{\gamma_2}])\bullet ([\tau_{V_2}(V_1)]^{m_2},[{\gamma_2}'])\\
        &=([\widetilde{V}_2]^{m_1}\cdot {\tau_{V_2}}^{-1}_*[\tau_{V_2}(V_1)]^{m_2},[{\gamma_2}][{\gamma_2}'])\\
        &=([\widetilde{V}_2^{m_1}\star \widetilde{V}_1^{m_2}],1)
    \end{align*}
    With respect to the geometric basis $(\gamma_1,\gamma_2)$, the loop $\alpha :=\widetilde{V}_2^{m_1}\star \widetilde{V}_1^{m_2}$ can be decomposed as $\alpha_1\star \tau_{V_1}^{-1}\alpha_2$ as follows
    \begin{align*}
        \alpha_1&:= \widetilde{V}_2^{m_1}\star \widetilde{V}_1^{m_2}\star \tau_{V_1}^{-1}(\widetilde{V}_2^{-m_1})\\
        \alpha_2 &:= \widetilde{V}_2^{m_1}
    \end{align*}
    and each $\alpha_i$ clearly satisfied the existence criterion (Theorem \ref{theorem criterion existence section LF over disk with several thimbles}) for a section $\sigma_i$ w.r.t. the original geometric basis. This shows that the property that the homotopy class $([\widetilde{V}_2^{m_1}\star \widetilde{V}_1^{m_2}],1)$ can be represented as the boundary of a section $\sigma\in \Gamma_0(E)$ is intrinsic to the loop and $E$, rather than the specific choice of vanishing paths.
\end{rem}

\subsection{Homologically distinct sections of some achiral Lefschetz fibrations over the sphere}\hfill\\

We give here a criterion for the existence of homologically distinct sections of the \textit{achiral} Lefschetz fibration $D(\pi):D(E)\to S^2$ arising as the double of a given Lefschetz fibration over the disk $\pi: E\to D$ with closed fibers and monodromy $\tau$. Recall that \[D(E) =E \bigsqcup_{\partial E} \overline{E},\] where the gluing takes place along the boundary using the identity map. The second copy inherits the opposite orientation hence the result is an achiral Lefschetz fibration.\\ 
Choose a factorization $\tau=\tau_{V_k}\circ \cdots \circ \tau_{V_1}$, let $\mathbf{m}=(m_1,\dots,m_k)\in \Z^k$. We define the based loop $V^{\textbf{m}}$ in $\Sigma$ as follows:
\[ V^{\textbf{m}}:= \widetilde{V}^{m_1}_1\star \tau_{V_1}^{-1}\widetilde{V}^{m_2}_2\star \cdots \star (\tau_{V_{k-1}}\circ \cdots \circ \tau_{V_1})^{-1}\widetilde{V}^{m_k}_k.\]
It is immediate that $([V^{\mathbf{m}}],1)$ satisfies the assumption of Theorem \ref{theorem criterion existence section LF over disk with several thimbles} and hence there exists a section $\sigma_{\textbf{m}}\in \Gamma(E)$ whose restriction is a representative for $([V^{\mathbf{m}}],1)$.\\
In the following, $[\cdot ]$ will denote an equivalence class in homology, rather than in homotopy. Since there will be no further mentions to homotopy classes this should not cause any confusion.
\begin{theorem}\label{theorem sections of doubles of LF over disk}
For any section $\sigma_{\textbf{m}}\in \Gamma_0(E)$ constructed as above, its double $D(\sigma_{\textbf{m}})$ is a section of $D(E)$. If there exists an index $j\in \{1,\dots,k\}$ and an integer $m$  such that \[m\cdot [V_j] \notin \mathrm{Im}(\Id_\ast-\tau_*)\subset H_1(\Sigma;\Z),\] then $D(E)$ admits at least two homologically distinct sections.
\end{theorem}
\begin{proof}
 The fact that a section $\sigma$ of $E$ gives rise to a section $D(\sigma)$ of $D(E)$ is immediate.\\
Consider the boundary map arising from the M-V sequence applied to the decomposition $D(E) = E \bigsqcup_{\partial E} \overline{E}$
        \[ \delta: H_2(D(E);\Z)\to H_1(\Sigma_{\tau};\Z)\cong \bigslant{H_1(\Sigma;\Z)}{\text{Im}(\Id_*-\tau_*)};\]
        two sections $s_1, s_2$ of $D(E)$ are homologically distinct if 
        \[ \delta([s_1])\neq \delta([s_2]).\]
The first section we consider is $D(\sigma_{\textbf{0}})$, i.e. the section associated with all $m_i=0$. Up to Hurwitz moves, we can assume that $j=1$ in the assumption of the Theorem; there is an integer $m$  such that \[m\cdot [V_1] \notin \mathrm{Im}(\Id_\ast-\tau_*).\] Consider the section $D(\sigma_{m\cdot{e_1}})$. The boundary map $\delta$ sends $D(\sigma_{\textbf{0}})$ to $0$, while it is easy to see via the Picard Lefschetz formula that $\delta\big(D(\sigma_{m\cdot{e_1}})\big)=m\cdot[V_1]\neq 0$ as an element of $H_1(\Sigma_{\tau};\Z)$, hence the two sections are homologically distinct.
\end{proof}
In the case of $\tau\sim \mathrm{Id}\in \mathrm{Aut}(\Sigma)$, then $0=\text{Im}(\Id_*-\tau_*)$ and by results of I. Smith (\cite{Smi99}) there must be at least one homologically essential vanishing cycle among the $V_i$'s. Hence we have the following corollary:
\begin{cor}\label{cor on theorem section doubles}
    Let $E\to D$ as above. If $\tau\sim \id$ in $\text{Aut}(\Sigma)$, then $D(E)$ admits countably many homologically distinct sections.
\end{cor}

\newpage
\bibliographystyle{alpha}
\bibliography{bibliography.bib}
\end{document}